\documentclass[11p]{article}

\usepackage{times}
\usepackage{amsthm}
\usepackage{amsmath}
\usepackage{amssymb}
\usepackage{mathrsfs}
\usepackage{pb-diagram}
\usepackage{xcolor}
\usepackage{hyperref}

\def\A{{\sf A}}
\def\B{{\sf B}}
\def\C{{\sf C}}

\def\th{\theta}
\def\Th{\Theta}
\def\N{\mathbb{N}}

\def\E{{\mathcal E}}

\def\S{\mathcal S}

\def\r{{\rm r}}
\def\d{{\rm d}}
\def\bu{\bullet}
\def\di{\diamond}

\def\({\left(}
\def\[{\left[}
\def\){\right)}
\def\]{\right]}

\def\si{\sigma}
\def\Si{\Sigma}
\def\G{{\sf G}}

\def\<{\langle}
\def\>{\rangle}
\def\tr{\triangleright}

\usepackage{tikz-cd} 

 \newtheorem{thm}{Theorem}[section]
 
 \newtheorem{lem}[thm]{Lemma}
 \newtheorem{prop}[thm]{Proposition}
 \theoremstyle{definition}
 \newtheorem{defn}[thm]{Definition}
 \theoremstyle{remark}
 \newtheorem{rem}[thm]{Remark}
 \newtheorem{ex}[thm]{Example}
 \numberwithin{equation}{section}

\numberwithin{equation}{section}

\begin{document}


\title{Reccurence Sets for Partial Inverse\\ Semigroup Actions and Related Structures}

\author{M. M\u antoiu}

\author{M. M\u antoiu
\footnote{
\newline
\textbf{Acknowledgements:} {The author has been supported by the Fondecyt Project 1200884. 
}}
}



\maketitle


\begin{abstract}
Two types of recurrence sets are introduced for inverse semigroup partial actions in topological spaces. We explore their connections with similar notions for related types of imperfect symmetries (prefix inverse semigroup expansions, partial group and groupoid actions).
\end{abstract}


\section{Introduction}\label{introduction}

In the theory of classical dynamical systems, a special role is played by {\it recurrence sets} (also called {\it dwelling sets}); this can be seen in any of the standard texts, as \cite{Au,dV} for example. Suppose that the group $\G$ acts continuously on the topological space $\Si$\,, the action being denoted by $\alpha$\,. For $M,N\subset\Si$ one defines 
\begin{equation}\label{indaia}
\G(\alpha)_M^N:=\big\{g\in\G\,\big\vert\,\alpha_g(M)\cap N\ne\emptyset\big\}\,.
\end{equation}
Such subsets of $\G$ are relevant for the behavior of the dynamical systems in many ways. As a basic example, for singletons $M=\{\si\}$ and $N=\{\tau\}$\,, the set $\G(\alpha)_\si^\tau\equiv\G(\alpha)_{\{\si\}}^{\{\tau\}}$ is non-void if and only if $\si$ and $\tau$ belong to the same orbit. This serves to characterize invariant subsets. Many other important properties may be defined or described in terms of recurrence sets; we mention only (topological) transitivity, limit sets, (non-)wandering points, periodic and almost periodic points, minimality, mixing, but there are others. The versatility of the notion comes both from the nature and the relative position of the sets $M,N$ and from the requirements on the size of the correspondent recurrent set.

\smallskip
To study more general and subtle notions of (local or partial) symmetry, in the last decades groups have been replaced by more general mathematical objects as groupoids or inverse semigroups. Even in the setting of groups, partial actions are an important way to encode imperfect symmetries. One important application (which has been our initial motivation) is towards the theory of $C^*$-algebras, but we will not refer to this in the present paper. 

\smallskip
Dynamical notions of such general type of actions have mostly been introduced in a non-systematic ad hoc manner. In \cite{FM} the topological dynamics of groupoid actions has been treated in a more consistent way, but only the basic theory has been developed. We are also interested in a similar project for inverse semigroup actions. In starting this project, we become aware of a certain intricacy of the topic and especially of the fact that well-known connections of the inverse semigroup actions with other mathematical notions may be a fruitful technical tool. Actually the present short article is devoted merely to explore these connections in the setting of recurrence, and the development of the theory itself is postponed to a further publication.

\smallskip
We are interested in triples $(\S,\th,\Si)$\,, where $\S$ is a (discrete) inverse semigroup, $\Si$ is a topological space and the action $\th$ is composed of a family of homeomorphisms $\{\th_s\!\mid\!s\in\S\}$ between open subsets of $\Si$\,, having suitable properties. The precise definition and the basic theory is presented in Section \ref{masinadecalcat}, but let us just note three peculiar features that are not present for groups: (a) The transformations $\th_s$ are not defined everywhere. (b) Two transformations $\th_s$ and $\th_t$ may act in the same way on an open subset of the intersection of their domains. (c) The composition $\th_s\circ\th_t$ has to be defined in a precise (maximal) way. We speak of a (genuine) action of $\S$ if this composition equals $\th_{st}$\,. Actually, to include the concept of partial group action, guided by \cite{BC}, we allow partial inverse semigroup actions, in which $\th_s\circ\th_t$ is only required to be a restriction of $\th_{st}$\,.

\smallskip  
By analogy to \eqref{indaia}, one can define recurrence sets of the form
\begin{equation}\label{pisica}
\S(\th)_M^N:=\big\{s\in\S\,\big\vert\,\th_s(M)\cap N\ne\emptyset\big\}\,.
\end{equation}
Since it is easier, and since sharp results will be obtained for this situation, let us restrict to the case when $M=\{\si\}$ is a one-point set. As a consequence of (a)\,, it is understood that only elements $s$ for which $\th_s$ is defined in $\si$ are candidates in \eqref{pisica}\,. On the other hand, because of (b), the set $\S(\th)_\si^N$ seems redundant. This could be important, since for some dynamical properties its size (infinite, complement of a finite set, syndetic) should play an important role. This is why, besides the "naive" set $\S(\th)_\si^N$ we also define in \eqref{reciurset} a quotient notion $\mathscr S(\th)_\si^N$, that is more suited to the developments we have in view in the future. Comparing these two types of recurrence sets with others, attached to related mathematical structures, is the main goal of the paper.

\smallskip
In \cite{Ex0}, to each group $\G$ one attaches an inverse semigroup $\S_{\G}$\,, suitably generated by $\G$\,, such that the partial actions of $\G$ transform into actions of $\S_\G$\,, the procedure satisfying a certain universal property. In \cite{LMS} and \cite{BE} this is extended to partial inverse semigroup actions: One starts with an inverse semigroup $\mathcal A$\,, generates another inverse semigroup $\S_\mathcal A$ satisfying certain relations, and then there is a one-to-one correspondence between the partial actions of $\mathcal A$ and the genuine actions of $\S_\mathcal A$\,. In Section \ref{cocoplast} we study the fate of orbits and recurrence sets under this correspondence of actions.

\smallskip
In Section \ref{cocoblast} we work with quotients $\S/_{\!\approx}$ of $\S$ through {\it idempotent pure congruences}. When $\S$ is $E$-unitary, and $\approx$ is the minimum group congruence, then $\S/_{\!\approx}$ is actually {\it the maximal homomorphic group image of} $\S$\,. By \cite{Kr}, there is a canonical way to transform a partial action of $\S$ into a partial action of $\S/_{\!\approx}$ in the same topological space (the case of genuine actions is simpler). We show that the two connected actions have the same invariant sets, and we find relations between the respective recurrence sets.

\smallskip
In Sections \ref{cafloryfier} and \ref{caflorifier} we use well-known connections between inverse semigroups and \'etale groupoids, looking for their effect on recurrence sets. In \cite{FM} a rather detailed study of recurrence and related dynamical properties for continuous groupoid actions can be found, so the results of these two sections will be combined in a future work to convert available groupoid information in terms of inverse semigroup action. But more work will be needed to have a comprehensive theory.

\smallskip
To an inverse semigroup partial action one associates its {\it groupoid of germs}; there is a perfect correspondence of invariant sets. Natural recurrence sets (see Definition \ref{ahu}) may be expressed in pure groupoid lenguaje, as (maybe non-invariant) reductions, as shown in Proposition \ref{aha}. Their connection with the two types of recurrence sets of the initial partial action are then subject of Propositions \ref{bucuros} and \ref{ohanesian}.

\smallskip
Section \ref{caflorifier}. For a groupoid action on a topological space, Definition \ref{giudat} indicates the natural recurrence sets. If the groupoid is \'etale, there is a canonical inverse semigroup formed of bisections. The actions of the groupoid generate actions of this inverse semigroup. Theorem \ref{vertij} and Proposition \ref{onegsion} indicate the way the relevant recurrece sets are linked.

\smallskip
The results on recurrent sets, in all the situations we treat, show a clear preference for $\mathscr S(\th)_\si^N$ upon $\S(\th)_\si^N$. The first set is each time in a one-to-one correspondence with the analog set of the related structure. The second one is only mapped surjectively. All the maps defining the connections seem interesting.

\newpage

\section{Inverse semigroup partial actions and their recurrence sets}\label{masinadecalcat}

For general facts on inverse semigroups and their partial actions we refer mainly to \cite{La,Pa,BC}. 

\begin{defn}\label{invsemigroup}
{\it An inverse semigroup} is a semigroup $\S$ such that for every $s\in\S$ there is a {\it unique} element $s^\sharp\in\S$ satisfying
\begin{equation*}\label{sharp} 
s^\sharp s s^\sharp=s^\sharp,\quad s s^\sharp s=s\,.
\end{equation*}
\end{defn}

We are always going to assume that $\S$ is discrete and infinite. In an inverse semigroup $\S$ a partial order relation is defined by
\begin{equation*}\label{por}
s\le t\quad{\rm if\ and\ only\ if}\quad s=ts^\sharp s\,.
\end{equation*}
On the commutative subsemigroup of idempotents $\mathcal E(\S)$ this reduces to $e=fe$\,. 
The order may be reformulated:
\begin{equation}\label{porg}
s\le t\,\Leftrightarrow\,\exists\,e\in\mathcal E(\S)\,,\,s=et\,\Leftrightarrow\,\exists\,f\in\mathcal E(\S)\,,\,s=tf\,.
\end{equation}

\smallskip
${\rm PHomeo}(\Si)$\,, the family of all the homeomorphisms between open subsets of $\Si$\,, is an inverse semigroup with obvious composition and inversion, the order relation being now just the restriction of functions. 

\begin{defn}\label{arhonti}
{\it A partial action} of the inverse semigroup $\S$ in the topological space $\Si$ is a family 
\begin{equation*}\label{draciuscati}
\th:=\Big\{\Si_{s^\sharp}\!\overset{\th_s}{\longrightarrow}\Si_s\,\Big\vert\,s\in\S\Big\}\,,
\end{equation*}
where $\th_s:\Si_{s^\sharp}\!\equiv{\rm dom}(\th_s)\!\to\Si_s\equiv{\rm im}(\th_s)$ is a homeomorphism between open subsets of $\Si$\,, satisfying for every $s,t\in\S$\,:
\begin{enumerate}
\item[(i)] $\th_{s^\sharp}=\th_s^{-1}$\,,
\item[(ii)] $\th_s\circ\th_t$ is a restriction of $\th_{st}$\,,  
\item[(iii)] if $s\le t$ then $\th_s$ is a restriction of $\th_t$\,,
\item[(iv)] $\Si=\bigcup_{e\in\mathcal E(\S)}\!\Si_e$ (non-degeneracy).
\end{enumerate}
\end{defn}

We speak of a {\it (genuine) action} whenever 
$$
\th_s\circ\th_t=\th_{st}\,,\,\forall\,s,t\in\S.
$$ 
Note that (ii) implies that
\begin{equation}\label{urmari}
\Si_t\cap\th_t\big(\Si_s\big)\subset\Si_{ts}\,,\quad\forall\,s,t\in\S,
\end{equation}
from which one deduces immediately
\begin{equation*}\label{urmariri}
\Si_{t^\sharp}\subset\Si_{t^\sharp t}\,,\quad\Si_{t}\subset\Si_{tt^\sharp}\,,\quad\forall\,t\in\S.
\end{equation*}
In addition, $\th_{e}={\rm id}_{\Si_e}$ for any $e\in\E(\S)$\,. We will often use the notation 
$$
\th_s(\si)\equiv s\diamond\si\equiv s\diamond_\th\!\si.
$$

We define now several dynamical notions. It is useful to set 
\begin{equation*}\label{cencep}
\S\{\th,\si\}:=\big\{s\in\S\,\big\vert\,\si\in\Si_{s^\sharp}\!={\rm dom}(\th_s)\big\}
\end{equation*} 
for the family of points of $\S$ which can be applied to $\si$. In virtue of \eqref{urmari}
\begin{equation*}\label{invirtue}
t\in\S\{\th,\si\}\,,\,s\in\S\{\th,t\di\si\}\,\Rightarrow\,st\in\S\{\th,\si\}\,.
\end{equation*}
An equivalence relation is defined in $\Si$ by
\begin{equation}\label{releq}
\si\overset{\th}{\sim}\tau\quad{\rm if\ and\ only\ if}\quad \exists\,s\in\S\{\th,\si\}\ \,{\rm with}\ \,\th_s(\si)=\tau.
\end{equation}
Consequently, one has usual notions as {\it orbit, orbit closure, invariant subset}, etc. 

\begin{defn}\label{grintesa}
Let $(\S,\th,\Si)$ be an inverse semigroup action. {\it The na\"ive recurrence set assigned to $M,N\subset\Si$} is
\begin{equation}\label{recurset}
\begin{aligned}
\S(\th)^N_M\!:&=\big\{s\in\S\,\big\vert\,\th_s(M)\cap N\ne\emptyset\big\}\\
&=\big\{s\in\S\,\big\vert\,\exists\,\si\in\Si_{s^\sharp}\cap M,\ \th_s(\si)\in N\big\}\,.
\end{aligned}
\end{equation}
\end{defn}

Clearly, if $s\in\S(\th)^N_M$ and $s\le t$ then $t\in\S(\th)^N_M$\,. In the case $M=\{\si\}$ and $N=\{\tau\}$ (with simplified notations)
$$
r,s,t\in\S(\th)^\tau_\si\,\Rightarrow s^\sharp\in\S(\th)_\tau^\si\ \,{\rm and}\ \,rs^\sharp t\in\S(\th)^\tau_\si\,.
$$
In particular $\S(\th)^\si_\si$ is an inverse subsemigroup of $\S$, and $\S(\th)^N_\si$ is contained in $\S\{\th,\si\}=\S(\th)_\si^\Si$ for every $N\subset\Si$\,.

\begin{defn}\label{chivalenta}
We introduce an equivalence relation on $\S\{\th,\si\}$ by
\begin{equation}\label{ivalenta}
s\overset{\th}{\underset{\si}\leftrightarrow} t\quad{\rm if\ and\ only\ if}\quad  \exists\,r\in\S\{\th,\si\}\,,\  r\le s,t\,.
\end{equation}
The quotient space is denoted by $\mathscr S\{\th,\si\}:=\S\{\th,\si\}/_{\!\overset{\th}{\underset{\si}\leftrightarrow}}$\,. The equivalence class of $s$ will be denoted by $\<s\>_\si^\th$ or by $\pi_\si^\th(s)$\,, giving rise to the quotient map 
\begin{equation}\label{quotientm}
\pi_\si^\th:\S\{\th,\si\}\to\mathscr S\{\th,\si\}\,.
\end{equation}
\end{defn}

\begin{rem}\label{aceea}
It follows from the definition that if $s\overset{\th}{\underset{\si}\leftrightarrow} t$ then $\th_s(\si)=\th_t(\si)$ (the converse fails in general). Concerning the connection between $\overset{\th}{\underset{\si}\leftrightarrow} $ and the minimum group congruence $\leftrightarrow$\,, see below.
\end{rem}

The next definition is meant to eliminate redundant contributions to recurrence.

\begin{defn}\label{scrintesa}
Let $(\S,\th,\Si)$ be an inverse semigroup action. {\it The recurrence set assigned to $\si\in\Si$ and $N\subset\Si$} is
\begin{equation}\label{reciurset}
\mathscr S(\th)^N_\si\!:=\S(\th)^N_\si\!/_{\!\overset{\th}{\underset\si\leftrightarrow}}=\big\{\<s\>^\th_\si\,\big\vert\,s\in\S(\th)^N_\si\big\}\,.
\end{equation}
\end{defn}

A standard notion \cite{La,Pa} is the following:

\begin{defn}\label{maxgrup}
{\it The maximum group homomorphic image} of the inverse semigroup $\S$ is the quotient $\S/_{\!\leftrightarrow}$ of $\S$ through {\it the minimum group congruence}
\begin{equation}\label{congruence}
s\leftrightarrow t\,\Leftrightarrow\,\exists\,r\in\S\,,\,r\le s,t\,.
\end{equation}
The class of $s\in\S$ is denoted by $\<s\>$ or by $\pi(s)$\,.
\end{defn}

\begin{rem}\label{notthesame}
The congruence $\leftrightarrow$ is intrinsic to $\S$. If $(\S,\th,\Si)$ is a partial action, on every set $\S\{\th,\si\}$ we have defined the equivalence relation $\overset{\th}{\underset{\si}\leftrightarrow}$ in \eqref{ivalenta}\,, which does depend on the given action. If $s,t\in\S\{\th,\si\}$ for some $\si$, then $s\overset{\th}{\underset{\si}\leftrightarrow} t$ implies $s\leftrightarrow t$\,, but the opposite implication generally fails, since in \eqref{congruence} $r$ is not bound to belong to $\S\{\th,\si\}$\,. In general the restriction of the equivalence relation $\leftrightarrow$ to $\S\{\th,\si\}$ is weaker than $s\overset{\th}{\underset{\si}\leftrightarrow} t$\,. However, see Lemma \ref{oareoare}.
\end{rem}

Recall that the inverse semigroup $\S$ is called {\it $E$-unitary} if every element larger than an idempotent must be an idempotent. This can be restated in several ways, for instance: if $e\leftrightarrow s$ and $e$ is idempotent, then $s$ must also be idempotent. For us, an important fact is the following result, that we did not find in the literature:

\begin{lem}\label{oareoare}
Let $(\S,\th,\Si)$ be the partial action of an $E$-unitary inverse semigroup. If $\si\in\Si$ and $s,t\in\S\{\th,\si\}$\,, then
\begin{equation*}\label{notari}
s\leftrightarrow t\,\Leftrightarrow\,s\overset{\th}{\underset{\si}\leftrightarrow} t\,.
\end{equation*}
\end{lem}

\begin{proof}
As mentioned in Remark \ref{notthesame}, we only need to check $\Rightarrow$\,. So let $s,t\in\S\{\th,\si\}$ such that $s\leftrightarrow t$\,. From \cite[Pag.\,25,\,66]{La} it follows that they are {\it compatible}, i.e. $s^\sharp t$ and $st^\sharp$ are idempotents, and $s\wedge t:=st^\sharp t=ts^\sharp s$ is the infimum of $s$ and $t$ with respect to $\le$\,. The domain of $\th_{s\wedge t}$ contains
$$
{\rm dom}\big(\th_s\circ\th_{t^\sharp}\circ\th_t\big)={\rm dom}\big(\th_s\big)\cap{\rm dom}\big(\th_t\big)=\Si_{s^\sharp}\cap\Si_{t^\sharp}\ni \si\,,
$$ 
so $\S\{\th,\si\}\ni s\wedge t\le s,t$\,, thus $s\overset{\th}{\underset{\si}\leftrightarrow} t$\,.
\end{proof}

\begin{rem}\label{maisafie}
Let $(\S,\th,\Si)$ be the partial action of an inverse semigroup and let $M,N\subset\Si$\,. Then the recurrence set $\mathscr S(\th)^N_M$ may be introduced directly as the quotient of the na\"ive recurrence set $\S(\th)^N_M$ through the minimum group congruence $\leftrightarrow$\,:
\begin{equation*}\label{poateiese}
\mathscr S(\th)^N_M\!:=\S(\th)^N_M/_{\!\leftrightarrow}\,.
\end{equation*}
If $\S$ is $E$-unitary, this may be written in terms of \eqref{reciurset} as
\begin{equation*}\label{poateiesse}
\mathscr S(\th)^N_M=\!\bigcup_{\si\in M}\!\mathscr S(\th)^N_\si.
\end{equation*}
\end{rem}

\section{Inverse semigroup expansions}\label{cocoplast}

Following \cite{BE}, we describe briefly {\it the prefix expansion inverse semigroup $\S_\mathcal A$ associated to an inverse semigroup} $\mathcal A$\,. A different point of view can be found in \cite{LMS}\,; the group case appeared in \cite{Ex0} (see also \cite{Ex1}). $\S_\mathcal A$ is generated by elements $[a]$\,, where $a\in\mathcal A$\,, subject to the relations
\begin{equation*}\label{racilaci}
(i)\ \,[a^\sharp][a][b]=[a^\sharp][ab]\,,\quad (ii)\ \,[a][b][b^\sharp]=[ab][b^\sharp]\,,\quad (iii)\ \,[a][a^\sharp][a]=[a]\,.
\end{equation*}
Buss and Exel show that there is an involutive anti-automorphism $\,^\#\!:\S_\mathcal A\to\S_\mathcal A$ such that $[a]^\#\!=[a^\sharp]$ for every $a\in\mathcal A$\,. In fact  $\S_\mathcal A$ is an inverse semigroup; it is $E$-unitary if and only if $\mathcal A$ is $E$-unitary (in this case Lemma \ref{oareoare} applies). The elements $\epsilon_a:=[a][a^\sharp]=[a][a]^\#$ are commuting idempotents. It is also shown that every element of $\S_\mathcal A$ has a unique {\it normal form}
\begin{equation}\label{stanform}
\epsilon_{c_1}\dots\epsilon_{c_n}[a]\,,\quad c_1,\dots c_n,a\in\mathcal A,\ n\in\N\,,
\end{equation}
under the conditions $c_1 c_1^\sharp=\dots=c_n c_n^\sharp=aa^\sharp$ and $a,aa^\sharp\in\{c_1,\dots,c_n\}$\,. It is useful to use the notations
\begin{equation*}\label{grabit}
\iota:\mathcal A\to\S_\mathcal A\,,\quad\iota(a):=[a]\,,
\end{equation*}
\begin{equation*}\label{hrabit}
q:\S_\mathcal A\to\mathcal A\,,\quad q\big(\epsilon_{c_1}\dots\epsilon_{c_n}[a]\big)\!:=a\,.
\end{equation*}
Clearly $\iota$ is a section of $q$\,.

\smallskip
The main property of $\S_\mathcal A$ is its {\it universal property}: If $\mathcal T$ is an inverse semigroup and  $\psi:\mathcal A\to\mathcal T$ is a map satisfying for all elements $a,b\in\mathcal A$ the conditions
\begin{equation*}\label{ditions}
\psi(a^\sharp)\psi(a)\psi(b)=\psi(a^\sharp)\psi(ab)\,,\quad\psi(a)\psi(b)\psi(b^\sharp)=\psi(ab)\psi(b^\sharp)\,,\quad\psi(a)\psi(a^\sharp)\psi(a)=\psi(a)\,,
\end{equation*}
there is a unique inverse semigroup morphism $\Psi:\S_\mathcal A\to\mathcal T$ such that $\Psi([a])=\psi(a)$ for each $a\in\mathcal A$\,.

\smallskip
It is then easy to see that there is a one-to-one correspondence between partial actions of the inverse semigroup $\mathcal A$ and (genuine) actions of its prefix expansion inverse semigroup $\S_\mathcal A$ (in the same topological space $\Si$). 

\begin{prop}\label{anuscata}
The two actions have the same invariant sets. 
\end{prop}

\begin{proof}
Start with the partial action $(\mathcal A,\vartheta,\Si)$ and denote by $\th:\S_\mathcal A\to{\rm PHomeo}(\Si)$ the associated inverse semigroup action. For an element in normal form \eqref{stanform} one has
\begin{equation}\label{distres}
\th_{\epsilon_{c_1}\dots\epsilon_{c_n}[a]}=\th_{\epsilon_{c_1}}\!\circ\dots\circ\th_{\epsilon_{c_n}}\!\circ\vartheta_a\,.
\end{equation}
Since $\epsilon_{c_k}$ is an idempotent, $\th_{\epsilon_{c_k}}$ is the identity map on its domain. Consequently, for every $\si\in\Si$ belonging to the (common) domain 
\begin{equation}\label{comun}
{\rm dom}\big(\th_{\epsilon_{c_1}\dots\epsilon_{c_n}[a]}\big)=\Si_{a^\sharp}\cap\vartheta_a^{-1}\big(\Si_{\epsilon_{c_1}}\cap\dots\cap\Si_{\epsilon_{c_n}}\big)
\end{equation} 
of the two transformations in \eqref{distres}, one has 
\begin{equation*}\label{dition}
\th_{\epsilon_{c_1}\!\dots\epsilon_{c_n}[a]}(\si)=\vartheta_a(\si)\,.
\end{equation*}
From this the conclusion can be deduced easily.
\end{proof}

Let us now relate the recurrence sets of the partial inverse semigroup action $\vartheta$ to the two types of recurrence sets of the associated prefix expansion inverse semigroup action $\th$. 

\begin{prop}\label{fusirat}
Let $\th:\S_\mathcal A\to{\rm PHomeo}(\Si)$ be the inverse semigroup action attached to the partial action $(\mathcal A,\vartheta,\Si)$ and let $M,N\subset\Si$\,. Then
\begin{equation}\label{abulic}
\iota\big[\mathcal A(\vartheta)_M^N\big]=\S_\mathcal A(\th)^N_M\cap\iota(\mathcal A)\,,
\end{equation}
\begin{equation}\label{babulic}
q\Big(\S_\mathcal A(\th)^N_M\Big)=\mathcal A(\vartheta)_M^N\,.
\end{equation}
\end{prop}

\begin{proof}
One has $s=[a]\in\S_\mathcal A(\th)^N_M\cap\,\iota(\mathcal A)$ if and only if there is some $\si\in M\cap\Si_{[a]^\#}\!=M\cap\Si_{a^\sharp}$ with $\th_{[a]}(\si)=\vartheta_a(\si)\in N$, and this means exactly that $[a]=\iota(a)\in\iota\big[\mathcal A(\vartheta)_M^N\big]$\,. Thus \eqref{abulic} is proven.

\smallskip
The relation \eqref{babulic} follows by analyzing the definitions, or we may write
$$
\begin{aligned}
q\Big(\S_\mathcal A(\th)^N_M\Big)&=q\Big(\S_\mathcal A(\th)^N_M\Big)\cap q\big(\iota(\mathcal A)\big)\\
&=q\Big(\mathcal S_\mathcal A(\th)^N_M\cap \iota(\mathcal A)\Big)\\
&=q\Big(\iota\big[\mathcal A(\vartheta)_M^N\big]\Big)\\
&=\mathcal A(\vartheta)_M^N\,.
\end{aligned}
$$
We used the fact that $\iota$ is a section of $q$\,, equation \eqref{abulic} and the fact that $\iota(\mathcal A)$ is saturated (to distribute $q$ to the intersection).
\end{proof}

\begin{thm}\label{potrivire}
Let $\th:\S_\mathcal A\to{\rm PHomeo}(\Si)$ be the inverse semigroup action attached to the partial action $(\mathcal A,\vartheta,\Si)$ and let $\si\in\Si$\,. 
\begin{enumerate}
\item[(i)]
The map $q:\S_\mathcal A\to\mathcal A$ restricts to a surjection $q_\si\!:\S_\mathcal A\{\th,\si\}\to\mathcal A\{\vartheta,\si\}$\,. 
\item[(ii)]
The map 
\begin{equation*}\label{map}
Q_\si:\mathscr S_\mathcal A\{\th,\si\}\!:=\S_\mathcal A\{\th,\si\}/_{\!\overset{\th}{\underset{\si}\leftrightarrow}}\to\mathcal A\{\vartheta,\si\}\,,\quad Q_\si\big[\pi^\th_\si(s)\big]:=q_\si(s)
\end{equation*}
is a well-defined bijection. One has bijectively
\begin{equation}\label{siguranta}
Q_\si\big[\mathscr S_\mathcal A(\th)_\si^N\big]=\mathcal A(\vartheta)_\si^N,\quad\forall\,N\subset\Si\,.
\end{equation}
\end{enumerate}
\end{thm}

\begin{proof}
$(i)$ As mentioned above, the domain of $\th_{\epsilon_{c_1}\dots\epsilon_{c_n}[a]}$ is $\Si_{a^\sharp}\cap\,\vartheta_a^{-1}\big(\Si_{\epsilon_{c_1}}\cap\dots\cap\Si_{\epsilon_{c_n}}\big)$\,, contained in  $\Si_{a^\sharp}$\,, the domain of $\vartheta_a$\,, where $a=q\big(\epsilon_{c_1}\dots\epsilon_{c_n}[a]\big)$\,. This shows that $q\big(\S_\mathcal A\{\th,\si\}\big)\subset\mathcal A\{\vartheta,\si\}$\,. Surjectivity is obvious: if $a\in\mathcal A\{\vartheta,\si\}$ then $[a]\in\S_\mathcal A\{\th,\si\}$ and $q([a])=a$\,.

\smallskip
$(ii)$ The statement about $Q_\si$ relies on showing that {\it $q_\si(s)=q_\si(t)$ if and only if $s\overset{\th}{\underset{\si}\leftrightarrow}t$}\,; this proves both the correctness of the definition and the injectvity. (Surjectivity follows from the fact that $q_\si$ is onto.) Using normal forms
$$
s:=\epsilon_{c_1}\dots\epsilon_{c_n}[a]\,,\quad t:=\epsilon_{d_1}\dots\epsilon_{d_m}[b]\,,
$$
the assumption $q_\si(s)=q_\si(t)$ reads $a=b$\,. Then 
\begin{equation*}\label{r}
r\!:=\epsilon_{c_1}\dots\epsilon_{c_n}\epsilon_{d_1}\dots\epsilon_{d_m}[a]\in\S_\mathcal A
\end{equation*} 
satisfies $r\le s,t$\,. From \eqref{comun} one gets easily 
$$
\si\in{\rm dom}\big(\th_r\big)={\rm dom}\big(\th_s\big)\cap{\rm dom}\big(\th_t\big)\,,
$$
so $r\in\S_\mathcal A\{\th,\si\}$ and we get $s\overset{\th}{\underset{\si}\leftrightarrow}t$\,. For the inverse implication, using \eqref{porg} and normal forms, one sees that $r\le s$ implies $q(r)=q(s)$\,. Then $s\overset{\th}{\underset{\si}\leftrightarrow}t$ implies $r\le s,t$ for some $r$, therefore $q(t)=q(r)=q(s)$\,. The relation \eqref{siguranta} follows from previous results and definitions:
$$
\mathcal A(\vartheta)_\si^N\overset{\eqref{babulic}}{=}q\Big(\S_\mathcal A(\th)^N_\si\Big)=Q_\si\Big[\pi^\th_\si\Big(\S_\mathcal A(\th)^N_\si\Big)\Big]\overset{\eqref{reciurset}}{=}Q_\si\Big[\mathscr S_\mathcal A(\th)^N_\si\Big]\,.
$$
\end{proof}

\section{Quotients through idempotent pure congruences}\label{cocoblast}

Let us fix an {\it idempotent pure congruence} $\approx$ in $\S$. By definition \cite{La}, this means that $\approx$ is an equivalence relation on $\S$ such that
\begin{equation*}\label{ameteala}
s\approx t\,,\ u\approx v\ \Rightarrow\ su\approx tv,
\end{equation*}
\begin{equation*}\label{abureala}
s\approx e\,,\ e\in\mathcal E(\S)\ \Rightarrow\ s\in\mathcal E(\S)\,.
\end{equation*}
We set $p(s)$ for the $\approx$-equivalence class of $s\in\S$\,. Note that $p:\S\to\S/_{\!\approx}$ is an inverse semigroup epimorphism, where the inversion in the quotient is $p(s)^\dagger\!:=p\big(s^\sharp\big)$\,. It is convenient to set $\mathcal R:=\S/_{\!\approx}$\,. A crucial fact (\cite[Lemma\,2.2]{Kr}) is that {\it if $(\S,\th,\Si)$ is a partial action of the inverse semigroup $\S$, there exists a unique partial action $\big(\mathcal R,\vartheta,\Si\big)$ such that}
\begin{itemize}
\item  $\Si^\vartheta_a=\bigcup_{p(s)=a}\Si^\th_s\,,\ \forall\,a\in\mathcal R$\,,
\item $\vartheta_{p(s)}(\si)=\th_s(\si)\,,\ \forall\,\si\in\Si^\th_{s^\sharp}\big(\!\subset\Si^\vartheta_{p(s)^\dagger}\big)$\,.
\end{itemize}

\begin{thm}\label{gramada}
Let $(\S,\th,\Si)$ and $\big(\mathcal R,\vartheta,\Si\big)$ as above.
\begin{enumerate}
\item[(i)] The two partial actions have the same orbits and the same invariant sets.
\item[(ii)] For every $\si\in\Si$ the map $p:\S\to\mathcal R$ restricts to a surjection $p_\si\!:\S\{\th,\si\}\to\mathcal R\{\vartheta,\si\}$\,. For $N\subset\Si$ one has $p_\si\big[\S(\th)_\si^N\big]=\mathcal R(\vartheta)_\si^N$.
\item[(iv)] The map 
\begin{equation*}\label{miap}
P_\si:\mathscr S\{\th,\si\}\!:=\mathcal S\{\th,\si\}/_{\!\overset{\th}{\underset{\si}\leftrightarrow}}\to\mathscr R\{\vartheta,\si\}:=\mathcal R\{\vartheta,\si\}/_{\!\overset{\vartheta}{\underset{\si}\leftrightarrow}}\,,\quad P_\si\big[\pi^\th_\si(s)\big]:=\pi^\vartheta_\si\big[p_\si(s)\big]
\end{equation*}
is a well-defined surjection. For $N\subset\Si$ one has 
\begin{equation}\label{setemare}
P_\si\big[\mathscr S(\th)_\si^N\big]=\mathscr R(\vartheta)_\si^N.
\end{equation}
\end{enumerate}
\end{thm}

\begin{proof}
$(i)$ follow easily from the definitions and from the way $\vartheta$ has been constructed out of $\th$\,. 

\smallskip
For $(ii)$, we have
$$
a\in\mathcal R\{\vartheta,\si\}\,\Leftrightarrow\,\si\in\Si^\vartheta_{a^\dagger}\!=\bigcup_{s\in a}\Si^\th_{s^\sharp}\,\Leftrightarrow\,\exists\,s\in a\,,\si\in\,\Si^\th_{s^\sharp}\,\Leftrightarrow\,\exists\,s\in\S\{\th,\si\}\,,\,p(s)=a\,. 
$$
Including the set $N$ in the arguments is trivial.

\smallskip
$(iii)$ It is clear that $P_\si$ is onto. One still has to check  that if $s,t\in\S\{\th,\si\}$ and $s\overset{\th}{\underset{\si}\leftrightarrow}t$ then $p_\si(s)\overset{\vartheta}{\underset{\si}\leftrightarrow}p_\si(t)$\,, ensuring that $P_\si$ is well-defined. This follows easily from the definition of the equivalences, since a homomorphism of inverse semigroups preserves the orderings. Then \eqref{setemare} is a consequence of the previous definitions and results:
$$
\mathscr R(\vartheta)_\si^N=\pi^\vartheta_\si\big[\mathcal R(\vartheta)^N_\si\big]=\pi^\vartheta_\si\big[p_\si\big(\S(\th)_\si^N\big)\big]=P_\si\big[\pi^\th_\si\big(\S(\th)_\si^N\big)\big]=P_\si\big[\mathscr S(\th)_\si^N\big]\,.
$$
\end{proof}

\begin{ex}\label{hazard}
Let us suppose now that $\S$ is a $E$-unitary inverse semigroup. It is known then \cite[Sect.\ 2.4]{La} that $\leftrightarrow$ is an idempotent pure congruence, that we use instead of $\approx$ in the above arguments. In addition $\mathcal R:=\S/_{\!\leftrightarrow}$ is a group (the maximum group homomorphic image of $\S$). To the partial action $\th$ of $\S$ on $\Si$ we associate as above the partial action $\vartheta$ of $\mathcal R$ in $\Si$\,. Since now $\mathcal R$ is a group, the equivalence relation $\overset{\vartheta}{\underset{\si}\leftrightarrow}$ on $\mathcal R\{\vartheta,\si\}$ is just the equality, so $\mathscr R\{\vartheta,\si\}=\mathcal R\{\vartheta,\si\}$. By Theorem \ref{gramada} one has the well-defined surjection $P_\si:\mathscr S\{\th,\si\}\to\mathscr R\{\vartheta,\si\}$\,. In fact, {\it it is even a bijection}. To show this, one has to check  for $s,t\in\S\{\th,\si\}$ that $p_\si(s)=p_\si(t)$ implies that $\pi^\th_\si(s)=\pi^\th_\si(s)$\,. Taking the definitions into consideration, this amounts to $s\leftrightarrow t\,\Rightarrow\,s\overset{\theta}{\underset{\si}\leftrightarrow}t$\,. This is solved by Lemma \ref{oareoare}. Consequently, in this case, {\it the recurrence sets $\mathscr S(\th)_\si^N$ and $\mathscr R(\vartheta)_\si^N$ are in a one-to-one correspondence for every $N\subset\Si$}\,.
\end{ex}

\section{From partial actions of inverse semigroups to groupoids}\label{cafloryfier}

{\it Groupoids} are small categories in which all the arrows are inverible. The object part of such a category $\Xi$ is also called {\it the unit space} and denoted by $\Xi^{(0)}\!\equiv X$.
{\it The source and range maps}, denoted by ${\rm d,r}:\Xi\to \Xi^{(0)}$, define the family of composable pairs
$$
\Xi^{(2)}\!=\{(\xi,\eta)\!\mid\!\d(\xi)=\r(\eta)\}\!\subset\Xi\times\Xi\,.
$$ 
For $M,N\subset X$ one sets 
\begin{equation}\label{faneaka}
\Xi_M\!:={\rm d}^{-1}(M)\,,\quad\Xi^N\!:={\rm r}^{-1}(M)\,,\quad\Xi_M^N:=\Xi_M\cap\Xi^N,
\end{equation}
with the particular notations $\Xi_x\equiv\Xi_{\{x\}}$ and $\Xi^x\equiv\Xi^{\{x\}}$. 
{\it A topological groupoid} is a groupoid $\Xi$ with a topology such that the inversion $\xi\mapsto\xi^{-1}$ and multiplication $(\xi,\eta)\mapsto \xi\eta$ are continuous. 

\smallskip
Following \cite{BC}, we now define and use {\it the groupoid of germs} associated to the partial action $\th$ of the inverse semigroup $\S$ on the topological space $\Si$\,. (For the case of genuine actions, see \cite{Ex,Ex1,Pa} and references therein.) In this generality, this also covers the transformation groupoids of partial group actions from \cite{Ex0}\,. First define
\begin{equation}\label{closegeam}
\begin{aligned}
\S\,\square\,\Si\!:&=\big\{(s,\si)\in\S\!\times\!\Si\,\big\vert\,\si\in\Si_{s^\sharp}\big\}\\
&=\big\{(s,\si)\in\S\!\times\!\Si\,\big\vert\,s\in\S\{\th,\si\}\big\}\,.
\end{aligned}
\end{equation}

\begin{defn}\label{caimak}
We say that $(s,\si),(t,\tau)\in\S\,\square\,\Si$ are {\it germ equivalent}, and write $(s,\si)\overset{\th}{\rightleftharpoons}(t,\tau)$\,, if $\si=\tau$ and $s\overset{\th}{\underset{\si}\leftrightarrow}t$\,.
\end{defn}

\begin{rem}\label{ttrecut}
Note that $\overset{\th}{\leftrightharpoons}$ is an equivalence relation, that depends on the partial action $\th$. However, when $\S$ is $E$-unitary, by Lemma \ref{oareoare}, it becomes independent of $\th$\,.
In the case of partial group actions, $\leftrightharpoons$ is just the equality.
\end{rem}

The quotient $\S\,\square\,\Si/_{\!\overset{\th}{\leftrightharpoons}}$ is also denoted as $\S\triangleright_\th\!\Si$\,.
Its generic elements are denoted by $\<s,\si\>\!:=[(s,\si)]_{\overset{\th}{\leftrightharpoons}}$\,. Then $\S\triangleright_\th\!\Si$ is a groupoid with unit space $\Si$\,, with the well-defined algebraic structure 
\begin{equation}\label{dezinvolt}
\mathfrak d\big(\<s,\si\>\big):=\si\,,\quad\mathfrak r\big(\<s,\si\>\big):=\th_s(\si)\,,
\end{equation}
\begin{equation*}\label{aialanta}
\big\<t,\th_s(\si)\big\>\<s,\si\>:=\<ts,\si\>\,,\quad\<s,\si\>^\star\!:=\big\<s^\sharp,\th_s(\si)\big\>\,.
\end{equation*}
Note that $\<e,\si\>$ is independent of $e\in\mathcal E(\S)$ if $\si\in\Si_e$\,, hence it will be identified with $\si$.

\smallskip
One defines a topology on $\S\triangleright_\th\!\Si$ as follows: Given $s\in\S$ and an open set $U\subset \Si_{s^\sharp}$, one sets 
\begin{equation*}\label{holosesc}
\<s,U\>:=\big\{\<s,\si\>\in\S\tr_\th\!\Si\,\big\vert\,\si\in U\big\}\,.
\end{equation*} 
The family of all these sets forms a basis for a topology that makes $\S\tr_\th\!\Si$ an \'etale topological groupoid. 

\smallskip
We are going to use systematically groupoid actions in the next section. One only needs here a particular case (the canonical action, cf.\;Example \ref{startlet}). For a topological groupoid $\Xi$ with unit space $\Si$\,, we write $\xi\ast\d(\xi):=\r(\xi)$ for every $\xi\in\Xi$\,. This may be reformulated as $\xi\ast\si:=\xi\si\xi^{-1}$ if $\si=\d(\xi)$\,. Orbit equivalence is denoted by $\si\overset{\Xi}{\sim}\tau$, and it means that $\si=\d(\xi)$ and $\tau=\r(\xi)$ for some $\xi\in\Xi$\,, i.e. $\Xi_\si^\tau\ne\emptyset$\,.

\begin{prop}\label{enoeno}
Use the shorthand $\,\Xi\equiv\S\tr_\th\!\Si$ and consider its canonical action on its unit space. One has $\si\overset{\Xi}{\sim}\tau$ if and only if $\si\overset{\th}{\sim}\tau$\,. Hence the orbits in the unit space of the germ groupoid coincide with the orbits of the initial inverse semigroup partial action. Invariant subsets are the same.
\end{prop}

\begin{proof}
By \eqref{releq}, $\si\overset{\th}{\sim}\tau$ if and only if $\th_s(\si)=\tau$ for some $s\in\S$ such that $\si\in\Si_{s^\sharp}$. Equivalently, $\<s,\si\>\ast\si=\tau$. Indeed, the only not-straightforward piece is the computation 
$$
\<s,\si\>\ast\si=\<s,\si\>\<s^\sharp s,\si\>\<s,\si\>^{-1}=\<ss^\sharp s,\si\>\<s^\sharp,\th_{s}(\si)\>=\<ss^\sharp,\th_{s}(\si)\>=\th_{s}(\si)=\tau
$$ 
and the proof is finished. The remaining assertions are direct consequences.
\end{proof}

Let us define 
\begin{equation*}\label{grosteta}
\Theta:\S\,\square\,\Si\to\Si\times\Si\,,\quad{\Th}(s,\si):=\big(\th_s(\si),\si\big)\,.
\end{equation*}
It is easy to see, by Remark \ref{aceea}, that the map
\begin{equation*}\label{grostata}
\widehat{\Th}:\S\tr_\th\!\Si\to\Si\times\Si\,,\quad\widehat{\Th}\big(\<s,\si\>\big):=\big(\th_s(\si),\si\big)
\end{equation*}
is well-defined. We get the diagram

\begin{equation}\label{coliflor}
\begin{diagram}
\node{\S\!\times\!\Si}\arrow{s,l}{p_1}\node{\S\,\square\,\Si}\arrow{sw,r}{\mathfrak p_1}\arrow{w,t}{j} \arrow{s,r}{\pi}\arrow{e,t}{{\Th}}\node{\Si\!\times\!\Si}\\ 
\node{\S}  \node{\S\tr_\th\!\Si}\arrow{ne,r}{\widehat{\Th}}
\end{diagram}
\end{equation}
where $j$ is an obvious embedding, $p_1$ is the first projection, ${\mathfrak p}_1$ its restriction to $\S\,\square\,\Si$ and $\pi$ is the quotient map.

\begin{defn}\label{ahu}
For $M,N\subset\Si$\,, the corresponding {\it groupoid recurrence set of the action $\th$} is $\widehat{\Th}^{-1}(N\!\times\!M)$\,.
\end{defn}

The pure groupoid interpretation of this set is seen in the first part of the following Proposition, while its connection with the na\"ive recurrence set of the inverse semigroup partial actions is subject of the second.

\begin{prop}\label{aha}
For every $M,N\subset\Si$\,, and using the notation \eqref{faneaka}, one has 
\begin{equation}\label{adam}
\big(\S\tr_\th\!\Si\big)^N_M=\widehat{\Th}^{-1}(N\!\times\!M)\,.
\end{equation}
On the other hand,
\begin{equation}\label{eva}
\S(\th)_M^N=\mathfrak p_1\big[{\Th}^{-1}(N\!\times\!M)\big]\,.
\end{equation}
\end{prop}

\begin{proof}
Taking into account \eqref{dezinvolt}, one has
$$
\begin{aligned}
\big(\S\tr_\th\!\Si\big)^N_M&=\big\{\<s,\si\>\in\S\tr_\th\!\Si\,\big\vert\,\mathfrak d\big(\<s,\si\>\big)=\si\in M,\,\mathfrak r\big(\<s,\si\>\big)=\th_s(\si)\in N\big\}\\
&=\big\{\<s,\si\>\in\S\tr_\th\!\Si\,\big\vert\,\big(\th_s(\si),\si\big)\in N\times M\big\}\\
&=\widehat{\Th}^{-1}(N\times M)\,,
\end{aligned}
$$
and \eqref{adam} is proven. For \eqref{eva}:
$$
\begin{aligned}
\mathfrak p_1\big[{\Th}^{-1}(N\!\times\!M)\big]&=\{s\in\S\!\mid\!\exists\,(s,\si)\in\S\square\Si\,,\,{\Th}(s,\si)\in N\!\times\!M\}\\
&=\{s\in\S\!\mid\!\exists\,(s,\si)\in\S\square\Si\,,\,\th_s(\si)\in N,\si\in M\}\\
&=\{s\in\S\!\mid\!\exists\,\si\in\Si_{s^\sharp}\cap M,\,\th_s(\si)\in N\}\\
&=\S(\th)_M^N\,.
\end{aligned}
$$
We used \eqref{closegeam} and \eqref{recurset}.
\end{proof}

\begin{rem}\label{bagaremarca}
Let $i:\S\tr_\th\!\Si\to\S\,\square\,\Si$ be a section of $\pi$ in diagram \ref{coliflor}. Setting 
$$
J:=\mathfrak p_1\circ i=p_1\circ j\circ i:\S\tr_\th\!\Si\to\S,
$$
it follows immediately that
$$
J\Big[\big(\S\tr_\th\!\Si\big)_M^N\Big]\subset\S(\th)_M^N\,.
$$
The inclusion is usually strict, and $J$ has no remarkable properties.
\end{rem}

For a one-point set $M=\{\si\}$ there is a direct way to relate germ groupoid recurrence sets with the inverse semigroup recurrent sets \eqref{reciurset}.  Taking into account Definition \ref{caimak}, the map
\begin{equation*}\label{atacoi}
\gamma_\si:\big(\S\tr_\th\Si\big)_\si\!\to\mathscr S\{\th,\si\}\,,\quad\gamma_\si\big(\<s,\si\>\big):=\<s\>^\th_\si=\pi^\th_\si(s)
\end{equation*}
is well-defined and bijective. From the definitions we deduce that, for every $N\subset\Si$\,, the map $\gamma_\si$ sends $\big(\S\tr_\th\Si\big)_\si^N$ to $\mathscr S(\th)_\si^N$\,. We formulate for further reference

\begin{prop}\label{bucuros}
For every $N\subset\Si$\,, the map $\big(\S\tr_\th\Si\big)_\si^N\ni\<s,\si\>\to\<s\>^\th_{\si}\in\mathscr S\{\th,\si\}$ restricts to a bijection between $\big(\S\tr_\th\!\Si\big)^N_\si$ and $\mathscr S(\th)_\si^N$\,.
\end{prop}

\begin{rem}\label{ciuci}
Actually one can also introduce
\begin{equation*}\label{atacoy}
\gamma:\S\tr_\th\Si\to\S/_{\!\leftrightarrow}\,,\quad\gamma\big(\<s,\si\>\big):=\<s\>=\pi(s)\,,
\end{equation*}
involving the maximal group homomorphic image from Definition \ref{maxgrup}. It is a well-defined surjection, but it is generally not one-to-one and it seems less convenient for our purposes.
\end{rem}

We indicate now a result that complements Proposition \ref{bucuros}. Recall that to a partial action $(\S,\th,\Xi)$ one associates the germ groupoid $S\tr_\th\!\Si$\,.

\begin{prop}\label{ohanesian}
Let us set $\,\Gamma:2^\S\to 2^{\S\tr_\th\Si}$\,, where
\begin{equation*}\label{samovar}
\Gamma(R):=\big\{\<s,\si\>\,\big\vert\,s\in R\,,\,\si\in\Si_{s^\sharp}\big\}\,.
\end{equation*}
For every $M,N\subset\Si$ one has 
\begin{equation*}\label{pompa}
\big(\S\tr_\th\!\Si\big)^N_M\subset\Gamma\big[\S(\th)^N_M\big]\,.
\end{equation*} 
The inclusion may be strict. For a one-point set $M=\{\si\}$\,, it is an equality:
\begin{equation*}\label{bompa}
\big(\S\tr_\th\!\Si\big)^N_\si\!=\Gamma\big[\S(\th)^N_\si\big]\,.
\end{equation*}
\end{prop}

\begin{proof}
For $s\in\S$ and $\si\in\Si$ we have
$$
\begin{aligned}
\<s,\tau\>\in\big(\S\tr_\th\!\Si\big)^N_M&\ \Leftrightarrow\ \tau\in\Si_{s^\sharp}\,,\ \mathfrak d\<s,\tau\>\in M\,,\ \mathfrak r\<s,\tau\>\in N\\
&\ \Leftrightarrow\ \tau\in\Si_{s^\sharp}\,,\ \tau\in M\,,\ \th_s(\tau)\in N\\
&\ \Rightarrow\ \tau\in\Si_{s^\sharp}\,,\ \th_s(M)\cap N\ne\emptyset\\
&\ \Leftrightarrow\ \tau\in\Si_{s^\sharp}\,,\ s\in\S(\th)^N_M\\
&\ \Leftrightarrow\ \<s,\tau\>\in\Gamma\big[\S(\th)^N_M\big]\,.
\end{aligned}
$$
At the third line, of course, there is just a one-sided implication: the condition $\th_s(M)\cap N\ne\emptyset$ has nothing to do with $\tau$. However, if $M$ is specified to be a singleton, one clearly gets an equality.
\end{proof}

\begin{rem}\label{cuci}
We indicate briefly the connection between Propositions \ref{ohanesian} and \ref{bucuros}. First, note that the map $\Gamma$ restricts to a surjection $\Gamma_\si\!:2^{\S\{\th,\si\}}\!\to 2^{(\S\tr_\th\Si)_\si}$, which in its turn gives rise to a well-defined bijection $\dot\Gamma_\si\!:2^{\mathscr S\{\th,\si\}}\!\to 2^{(\S\tr_\th\Si)_\si}$ given by
\begin{equation}\label{gargarita}
\dot\Gamma_\si\big[\pi^\th_\si(R)\big]:=\Gamma_\si(R)=\gamma_\si^{-1}\big[\pi^\th_\si(R)\big]\,,\quad\forall\,R\subset\S\{\th,\si\}\,.
\end{equation}
In \eqref{gargarita}, the first equality is the definition. The second equality, involving an inverse image, is the desired connection, easy to prove. To summarize, we have $\dot\Gamma_\si=\gamma_\si^{-1}$, where the right hand side should be seen as a set function.
\end{rem}

\section{From groupoid actions to inverse semigroup actions}\label{caflorifier}

\begin{defn}\label{grupact}
{\it A continuous groupoid action} $(\Xi,\rho,\kappa,\Si)$ consists of a topological groupoid $\Xi$\,, a topological space $\Si$\,, a continuous surjective map $\rho:\Si\rightarrow X$ ({\it the anchor}) and the continuous action map
\begin{equation*}\label{ganchor}
\kappa:\Xi\!\Join\!\Si:=\{(\xi,\si)\!\mid\!\d(\xi)=\rho(\si)\}\ni(\xi,\si)\mapsto{\kappa_\xi(\si)\equiv\xi\!\bu_\kappa\!\si}\in\Si
\end{equation*} 
satisfying the axioms: 
\begin{enumerate}
\item $\rho(\si)\!\bu_\kappa\!\si=\si,\, \forall \,\si\in \Si$\,,
\item if $(\xi,\eta)\in \Xi^{(2)}$ and $(\eta,\si)\in \Xi\!\Join\!\Si$, then $(\xi,\eta\!\bu_\kappa\!\si)\in \Xi\!\Join\!\Si$ and $(\xi\eta)\!\bu_\kappa\!\si=\xi\!\bu_\kappa\!(\eta\!\bu_\kappa\!\si)$\,.
\end{enumerate} 
If the action $\kappa$ is understood, we will write $\xi\bu\si$ instead of $\xi\bu_\kappa\si$. The theory is developed in \cite{Wi}.
\end{defn}

\begin{ex}\label{startlet}
Each topological groupoid acts continuously in a canonical way on its unit space: $\xi$ sends $\d(\xi)$ into $\r(\xi)$\,. Here $\Si=X$ and $\rho={\rm id}_X$, and then (special notation) $\xi\!\ast\!x:=\xi x\xi^{-1}$ if $\d(\xi)=x$. 
\end{ex}

\begin{ex}\label{valtoare}
The topological groupoid $\Xi$ also acts on itself, with $\Si:=\Xi$\,, $\rho:=\r$ and $\xi\bu\eta:=\xi\eta$\,.
\end{ex}

For $\xi\in\Xi\,,\,{\sf A},{\sf B}\subset\Xi\,,\,M\subset\Si$ we use the notations
\begin{equation}\label{otations}
{\sf A}{\sf B}:=\big\{\xi\eta\,\big\vert\,\xi\in{\sf A}\,,\,\eta\in{\sf B}\,,\,\d(\xi)=\r(\eta)\big\}\,,
\end{equation}
\begin{equation*}\label{otattions}
{\sf A}\bu M:=\big\{\xi\bu\si\,\big\vert\,\xi\in{\sf A}\,,\si\in M\,, \d(\xi)=\rho(\si)\big\}=\bigcup_{\xi\in{\sf A}}\xi\bu M\,.
\end{equation*}
A subset $M\subset \Si$ is called {\it invariant} if $\xi \bu M\subset M$, for every $\xi\in\Xi$\,. Particular cases are {\it the orbits} $\mathfrak O_\si\!:=\Xi_{\rho(\si)}\!\bu\si$ and {\it the orbit closures} $\overline{\mathfrak O}_\si$\,.  {\it The orbit equivalence relation} will be denoted by $\overset{\kappa}{\sim}$ or by $\overset{\bu}{\sim}$\,. 

\begin{defn}\label{giudat}
For every $M,N\subset\Si$ one defines {\it the recurrence set} 
\begin{equation*}\label{rrecur}
\widetilde\Xi_M^N=\{\xi\in\Xi\!\mid\!(\xi\bu M)\cap N\ne\emptyset\}\,.
\end{equation*}
\end{defn}

\begin{rem}\label{iverse}
Note that 
$\widetilde\Xi_M^N\subset\Xi_{\rho(M)}^{\rho(N)}$\,. When $\rho$ is also injective, one has $\,\widetilde\Xi_M^N=\Xi_{\rho(M)}^{\rho(N)}$\,. In Exemple \ref{startlet} one has $\widetilde\Xi_M^N=\Xi_M^N=\{\xi\in\Xi\!\mid\!\d(\xi)\in M,\,\r(\xi)\in N\}$\,, making the connection with the previous section.
\end{rem}

{\it A bisection} of the topological groupoid $\Xi$ is an open subset of the groupoid on which the restrictions of both $\d$ and $\r$ are injective. We recall that the groupoid $\Xi$ is called {\it \'etale} if $\d:\Xi\to\Xi$ is a local homeomorphism. If $\Xi$ is \'etale, $X\equiv\Xi^{(0)}$ is a clopen (closed and open) subset of $\Xi$\,, all the fibres $\Xi^x$ and $\Xi_x$ are discrete, and $\d,\r$ and the multiplication are open maps. 

\smallskip
Let $\Xi$ be an \'etale groupoid over the unit space $X$ and let ${\rm Bis}(\Xi)$ be {\it the inverse semigroup associated to} $\Xi$\,, formed of bisections (they form a basis of clopen sets for the topology of $\Xi$)\,. The multiplication is given by \eqref{otations} and the inverse is $\A^\sharp\!:=\A^{-1}$, under which
\begin{equation*}\label{holds}
\A^{-1}\A=\d(\A)\,,\quad\A\A^{-1}=\r(\A)
\end{equation*}
hold. The family of idempotents $\mathcal E[{\rm Bis}(\Xi)]={\rm Top}(X)$ is formed exactly of the open subsets of the unit space, $X$ being the unit and $\emptyset$ the zero element. 

\begin{defn}\label{above}
Let $(\Xi,\rho,\kappa,\Si)$ be a groupoid action, with $\Xi$ \'etale. The inverse semigroup ${\rm Bis}(\Xi)$ acts continuously on $\Si$ by
$$
{\rm dom}(\th_\A)\equiv\Si_{\A^{-1}}\!:=\rho^{-1}[\d(\A)]\,,\quad{\rm im}(\th_\A)\equiv\Si_{\A}\!:=\rho^{-1}[\r(\A)]\,,\quad
$$
\begin{equation}\label{bajbai}
\th_\A(\si):=\xi^\A_{\rho(\si)}\!\bu_\kappa\si\equiv\kappa_{\xi^\A_{\rho(\si)}}\!(\si)\,,\quad\forall\,\A\in{\rm Bis}(\Xi)\,,\ \si\in\rho^{-1}[\d(\A)]\,,
\end{equation}
where $\xi^\A_x=\d|_{A}^{-1}(x)$ denotes the unique element $\xi\in\A$ such that $\d(\xi)=x\in\A\subset X$. The action is genuine.
\end{defn}

\begin{prop}\label{oneon}
The orbits and the invariant sets in $\Si$ of the action $\kappa$ of the groupoid $\Xi$ coincide with those of the associated action $\th$ of $\,{\rm Bis}(\Xi)$ in $\Si$\,. 
\end{prop}

\begin{proof}
We will prove that $\si\overset{\kappa}{\sim}\tau\ \Leftrightarrow\ \si\overset{\th}{\sim}\tau$. We have $\si\overset{\th}{\sim}\tau$ if and only if $\exists\,\A\in{\rm Bis}(\Xi)$ with $\si\in\Si_{\A^{-1}}$ and $\th_\A(\si)=\tau$. By definition, this means that $\rho(\si)\in\d(\A)$ and $\xi_{\rho(\si)}^\A\bu\si=\tau$. This is equivalent with $\si\overset{\kappa}{\sim}\tau$, taking into account the defining property of a bisection. The remaining assertion of the proposition follows from this.
\end{proof}

To simplify, we are going to use the notation $\mathcal B:={\rm Bis}(\Xi)$\,. Thus there are also notations as $\mathcal B\{\th,\si\}$\,, $\mathscr B\{\th,\si\}$\,, for instance. We recall the quotient map \eqref{quotientm}, that in our concrete situation reads 
\begin{equation*}\label{ianuarie}
\pi^\th_\si:\mathcal B\{\th,\si\}\to\mathscr B\{\th,\si\}:=\mathcal B\{\th,\si\}/_{\!\overset{\th}{\underset\si\leftrightarrow}}\,.
\end{equation*}
We are looking for connections between different types of recurrence sets, attached to the actions, cf. Definitions \ref{grintesa}, \ref{scrintesa} and \ref{giudat}. 

\begin{thm}\label{vertij}
The map 
\begin{equation*}\label{vijelie}
\dot\delta_\si:\mathscr B\{\th,\si\}\to\widetilde\Xi_\si^\Si=\Xi_{\rho(\si)}\,,\quad\dot\delta_\si\big[\pi^\th_\si(\A)\big]:=\xi^\A_{\rho(\si)}
\end{equation*}
is a well-defined bijection. For every $N\subset\Si$ one has
\begin{equation}\label{contaition}
\dot\delta_\si\big[\mathscr B(\th)_\si^N\big]=\widetilde\Xi_\si^N.
\end{equation}
\end{thm}

\begin{proof}
First set 
\begin{equation*}\label{vitejie}
\delta_\si:\mathcal B\{\th,\si\}\to\Xi_{\rho(\si)}\,,\quad\delta_\si(\A):=\xi^\A_{\rho(\si)}\,,
\end{equation*}
which is onto, since the bisections cover the groupoid. Let $\A,\B\in\mathcal B\{\th,\si\}$\,, which means that $\A,\B$ are bisections and $\rho(\si)\in\d(\A)\cap\d(\B)$\,. The first statement of the Proposition follows if we check that
\begin{equation}\label{titejie}
\delta_\si(\A)=\delta_\si(\B)\,\Leftrightarrow\ \A\overset{\th}{\underset{\si}\leftrightarrow}\B\,.
\end{equation}
It is easy to see for two bisections that $\C\le{\sf D}$ if and only if $\C\subset{\sf D}$\,.  Using the definitions, we see that \eqref{titejie} reads in detail
\begin{equation}\label{titie}
\xi^\A_{\rho(\si)}\!=\xi^\B_{\rho(\si)}\,\Leftrightarrow\ \exists\,{\sf C}\subset\A\cap\B\ {\rm with}\ \rho(\si)\in\d({\sf C})\,.
\end{equation}
The implication $\Rightarrow$ in \eqref{titie} is solved by taking ${\sf C}:=\A\cap\B\ni\xi^\A_{\rho(\si)}\!=\xi^\B_{\rho(\si)}$\,. The other implication is trivial, by the unicity involved in the definition of the bisection $\C$\,.
Then \eqref{contaition} is a consequence of \eqref{bajbai} and of the definitions of the recurrence sets $\mathscr B(\th)_\si^N$ and $\widetilde\Xi_\si^N$.
\end{proof}

The next result yields a second point of view upon the connection between recurrence sets. We set 
\begin{equation*}\label{sigmavar}
\Delta:2^\Xi\to 2^{\mathcal B},\quad\Delta(E):=\{\A\in\mathcal B\!\mid\! \A\cap E\ne\emptyset\}\,.
\end{equation*} 

\begin{prop}\label{onegsion}
Let $(\Xi,\rho,\kappa,\Si)$ be an \'etale groupoid action and denote by $\big(\mathcal B,\th,\Si\big)$ the associated inverse semigroup partial action, where $\mathcal B={\rm Bis}(\Xi)$\,.
Let $\si\in\Si$ and $M,N\subset\Si$\,.
\begin{enumerate}
\item[(i)] One has
\begin{equation}\label{ciudatel}
\mathcal B(\th)^N_M=\Delta\big(\widetilde\Xi^N_M\big)\,.
\end{equation}
\item[(ii)] Let us denote by $\Delta_\si:2^{\Xi_{\rho(\si)}}\to 2^{\mathcal B\{\th,\si\}}$ the restriction of $\Delta$ to $2^{\Xi_{\rho(\si)}}$. Then $\Delta_\si=\delta_\si^{-1}$ (the last one regarded as a set function).
\end{enumerate}
\end{prop}

\begin{proof}
$(i)$ For $\A\in\mathcal B$ one may write
$$
\begin{aligned}
\A\in\mathcal B(\th)^N_M &\ \Leftrightarrow\ \th_A(M)\cap N\ne\emptyset\\
&\ \Leftrightarrow\ \exists\ \si\in M\cap\rho^{-1}[\d(\A)]\,,\ \th_\A(\si)\in N\\
&\ \Leftrightarrow\ \exists\ \si\in M\cap\rho^{-1}[\d(\A)]\,,\ \xi^\A_{\rho(\si)}\!\bu_\kappa\!\si\in N\\
&\ \Leftrightarrow\ \exists\ \si\in M\cap\rho^{-1}[\d(\A)]\,,\ \exists\,\xi\in\A\cap\Xi_{\rho(\si)} \,,\ \xi\bu_\kappa\!\si\in N\\
&\ \Leftrightarrow\ \exists\,\xi\in\A\,,\ \big(\xi\bu_\kappa\!M\big)\cap N\ne\emptyset\\
&\ \Leftrightarrow\ \A\cap\widetilde\Xi^N_M\ne\emptyset\\
&\ \Leftrightarrow\ \A\in\Delta\big(\widetilde\Xi^N_M\big)\,.
\end{aligned}
$$
The forth equivalence is due to the uniqueness of an element $\xi$ of the bisection $\A$ such that $\d(\xi)=\rho(\si)$\,.

\smallskip
$(ii)$ One gets as a particular case of \eqref{ciudatel} the relation 
\begin{equation*}\label{greu}
\mathcal B\{\th,\si\}=\mathcal B(\th)_\si^\Si=\Delta\big(\,\widetilde\Xi_\si\big)=\Delta\big(\,\Xi_{\rho(\si)}\big)\,.
\end{equation*}
This and the fact that $\Delta$ is increasing justifies the restriction $\Delta_\si$ in the diagram. If $E\in 2^{\Xi_{\rho(\si)}}$ then for every $\eta\in E$ one has $\d(\eta)=\rho(\si)$\,, which justifies the non-trivial part of the  third step below:
$$
\begin{aligned}
\delta_\si^{-1}(E)&=\big\{{\sf A}\in\mathcal B\{\th,\si\}\,\big\vert\,\delta_\si({\sf  A})\in E\big\}\\
&=\big\{{\sf A}\in\mathcal B\{\th,\si\}\,\big\vert\,\xi^{\sf A}_{\rho(\si)}\!\in E\big\}\\
&=\big\{{\sf A}\subset\mathcal B\{\th,\si\}\,\big\vert\,{\sf A}\cap E\ne\emptyset\big\}\\
&=\Delta_\si(E)\,.
\end{aligned}
$$
\end{proof}


\bigskip

M. M\u antoiu:

\smallskip
Departamento de Matem\'aticas, Universidad de Chile, 

Las Palmeras 3425, Casilla 653, Santiago, Chile

\smallskip
E-mail: mantoiu@uchile.cl

\end{document}